\documentclass[11pt]{article}
\usepackage{amsmath,amssymb,amsfonts,amsthm}
\usepackage[colorlinks,linktocpage,linkcolor=blue]{hyperref}
\usepackage{cleveref}
\usepackage{color}
\usepackage{epsfig,subfigure,epstopdf,float}
\usepackage[]{graphics}
\usepackage{graphicx,inputenc}
\usepackage{subfigure}
\usepackage[justification=centering]{caption}
\usepackage{wrapfig}
\usepackage{multirow}
\usepackage{booktabs}
\usepackage{algorithmic}

\usepackage[T1]{fontenc}
\usepackage{authblk}
\usepackage[textwidth=6.0in,textheight=9in]{geometry}
\usepackage[linesnumbered,ruled,vlined, ruled,vlined]{algorithm2e}

\usepackage[colorlinks,linktocpage,linkcolor=blue]{hyperref}
\usepackage{fancyhdr}

\usepackage[export]{adjustbox}

\newtheoremstyle{exampstyle}
  {\topsep} 
  {\topsep} 
  {\itshape} 
  {} 
  {\bfseries} 
  {.} 
  {.5em} 
  {} 

\theoremstyle{exampstyle}
\numberwithin{equation}{section}
\newtheorem{theorem}{Theorem}
\newtheorem{lemma}{Lemma}[section]

\newtheorem{corollary}[lemma]{Corollary}

\newtheorem{example}{Example}
\allowdisplaybreaks
\usepackage{parskip}
\let\oldref\ref
\renewcommand{\ref}[1]{(\oldref{#1})}  
\renewcommand{\eqref}[1]{(\oldref{#1})}

\newcommand{\ds}{\displaystyle}

\newbox\boxaddrone \newbox\boxaddrtwo

\newcommand{\N}{\mathcal{N}}

\newcommand{\ba}{\begin{eqnarray*}}
\newcommand{\ea}{\end{eqnarray*}}
\newcommand{\be}{\begin{equation}}
\newcommand{\ee}{\end{equation}}
\newcommand{\bea}{\begin{eqnarray}}
\newcommand{\eea}{\end{eqnarray}}

\allowdisplaybreaks
\newcommand{\p}{\partial}
\newcommand{\om}{\Omega}

\newbox\boxaddrone \newbox\boxaddrtwo

\def\N+{n\in\mathbb{N}^{+}}

\def\1d{\mathcal{D}((-\Delta)^{\gamma_1+1/2})}
\def\2d{\mathcal{D}((-\Delta)^{\gamma_2+1})}

\begin{document}
\title{Inverse source problem of sub-diffusion of variable exponent}
\author[1]{Zhiyuan Li}
\author[2,3]{Chunlong Sun\thanks{Author to whom any correspondence should be addressed, sunchunlong@nuaa.edu.cn}}
\author[4]{Xiangcheng Zheng}
\affil[1]{School of Mathematics and Statistics, Ningbo University, Ningbo 315211, Zhejiang, China}
\affil[2]{School of Mathematics, Nanjing University of Aeronautics and Astronautics, Nanjing 211106, Jiangsu, China}
\affil[3]{Nanjing Center for Applied Mathematics, Nanjing 211135, Jiangsu, China}
\affil[4]{School of Mathematics, Shandong University, Jinan 250100, Shandong, China}

\maketitle

\begin{abstract}
This work investigates both direct and inverse problems of the variable-exponent sub-diffusion model, which attracts increasing attentions in both practical applications and theoretical aspects. Based on the perturbation method, which transfers the original model to an equivalent but more tractable form, the analytical extensibility of the solutions and the weak unique continuation principle are proved, which results in the uniqueness of the inverse space-dependent source problem from local internal observation. Then, based on the variational identity connecting the inversion input data with the unknown source function, we propose a weak norm and prove the conditional stability for the inverse problem in this norm. The iterative thresholding algorithm and Nesterov iteration scheme are employed to numerically reconstruct the smooth and non-smooth sources, respectively. Numerical experiments are performed to investigate their effectiveness. \\

\noindent Keywords: inverse source problem, variable-exponent fractional equation, weak unique continuation, uniqueness.\\

\noindent AMS Subject Classifications: 35R11, 35R30, 65M32.
\end{abstract}

\section{\large Introduction}
Sub-diffusion model has been widely used in various fields such as anomalous diffusion \cite{DenHou}, where the exponent of the power function kernel in the sub-diffusion model is determined by the sublinear growth of the mean square displacement (MSD) of the particle with respect to time $t$ \cite{Che}. In \cite{SunChe} and \cite{SunChe2009}, the best-fitting approach based on MSD data on calcium ion diffusion in dendrites is applied to show that the sublinear growth rate is a linear function of $t$, resulting in the variable-exponent sub-diffusion model. Extensive applications of the variable-exponent sub-diffusion could be found in the review literature \cite{SunCha}. Mathematically, the work \cite{ZheLi} employs the variable exponent to perform a local modification of the sub-diffusion model. Such modification resolves the incompatibility between the nonlocal operators in subdiffusion and the local initial conditions, and thus eliminates the initial singularity of the solutions of the sub-diffusion, while retaining
its heavy-tail behavior away from the initial time. Consequently, the variable-exponent sub-diffusion model exhibits great potential in both practical applications and theoretical aspects.

Based on the above discussions, we consider the variable-exponent sub-diffusion equation \cite{SunChe,SunCha}
\begin{equation}\label{Prob-IBVP}
\begin{cases}
^c\p_t^{\alpha(t)} u(x,t)-  \Delta u (x,t)= F(x,t), & (x,t)\in\Omega_T:=\Omega\times(0,T),\\
u(x,0) = u_0(x),& x\in\Omega,\\
u(x,t) = 0, & (x,t)\in\partial\Omega\times(0,T).
\end{cases}
\end{equation}
Here $\Omega \subset \mathbb{R}^d$ with $1 \le d \le 3$ is a simply-connected bounded domain with the piecewise smooth boundary $\p \om$ with convex corners, $0<T<\infty$ is the terminal time, $x := (x_1,\cdots,x_d)$ denotes the spatial variables, $F$ and $u_0$ refer to the source term and the initial value, respectively, and the fractional derivative of order $0<\alpha(t)<1$ is defined via the variable-exponent Abel kernel $k$ \cite{LorHar}
\begin{equation*}
^c\p_t^{\alpha(t)}u :=k*\p_t u,~~k(t):=\frac{t^{-\alpha(t)}}{\Gamma(1-\alpha(t))} .
\end{equation*}

There exist sophisticated results on direct and inverse problems for sub-diffusion of constant exponent \cite{JinKia,Kal,KiaLiu,SakYam,Slo,ZhaZho} and spatially dependent variable exponent \cite{Hon,KiaSocYam,LeSty}. For the time-dependent case like (\ref{Prob-IBVP}), \cite{KiaSlo,Uma} focus on the piecewise constant variable exponent such that the solution representation is available for each temporal piece.  For the case where $\alpha(t)$ is smooth, two methods (i.e. the convolution method and the perturbation method) have been proposed in \cite{Zhe} to convert (\ref{Prob-IBVP}) into more tractable forms such that the well-posedness of (\ref{Prob-IBVP}) and its regularity of the solution could be analyzed. For inverse problems of (\ref{Prob-IBVP}), there exist rare studies and a recent result is the determination of $\alpha_0:=\alpha(0)$ from the solution data, where $\alpha_0$ is the initial value of $\alpha(t)$ that plays the key role in characterizing the properties of the solution, see \cite{Zhe}. Thus, both direct and inverse problems of (\ref{Prob-IBVP}) are far from well developed, which motivates the current study.

This work investigates both direct and inverse problems of the variable-exponent sub-diffusion model (\ref{Prob-IBVP}). Based on the perturbation method proposed in \cite{Zhe}, which transfers (\ref{Prob-IBVP}) to an equivalent but more tractable form, subtle estimates over a complex plane are made to show the analytical extensibility of the solutions from a finite time interval to an infinite time interval. Then the weak unique continuation principle is proved based on the Laplace transform method, which results in the uniqueness of the inverse space-dependent source problem from either local internal observation or partial boundary observation. Furthermore, based on the variational identity connecting the inversion input data with the unknown source function, we propose a weak norm and prove the conditional stability for the inverse problem in this norm. Finally, the iterative thresholding algorithm and Nesterov iteration scheme are used to numerically recover the source function, and numerical experiments are carried out to substantiate its effectiveness. 

The rest of this work is organized as follows. In Section \ref{sec2}, we prove the analytical extensibility and the weak unique continuation principle. In Section \ref{sec3} we prove the uniqueness and conditional stability of the inverse source problem. In Section \ref{sec4} we propose numerical methods and examples for reconstructing the source function.

\section{\large Well-posedness and weak unique continuation property for forward problem}\label{sec2}
\subsection{Notations}
For subsequent analysis, we first recall some standard results for parabolic equations. Let $L^p(\Omega)$ with $1 \le p \le \infty$ be the Banach space of $p$th power Lebesgue integrable functions on $\om$. Denote the inner product of $L^2(\Omega)$ by $(\cdot,\cdot)$. For a positive integer $m$,
let $ W^{m, p}(\Omega)$ be the Sobolev space of $L^p$ functions with weakly derivatives of $m$ in $L^p(\Omega)$. Let $H^m(\Omega) := W^{m,2}(\Omega)$ and $H^m_0(\Omega)$ be its subspace with the zero boundary condition up to the order $m-1$. For $\frac{1}{2}<s\leq 1$, $H^s_0(\Omega)$ is defined by interpolating $L^2(\Omega)$ and $H^1_0(\Omega)$ \cite{AdaFou}. 

For a Banach space $X$, we say that $f\in L^p(0,T;X)$ provided that
\begin{eqnarray*}
	\|{f}\|_{L^p(0,T;X)}:=
	\begin{cases}
		\ds\left(\int_0^T \|{f(\cdot,t)}\|_X^p {\rm d}t\right)^{1/p} <\infty,&  1\leq p <\infty,\\[0.1in]
		\ds{\rm ess\; sup}_{0<t<T} \|{f(\cdot,t)}\|_X <\infty, &  p=\infty.
	\end{cases}
\end{eqnarray*}
Let $W^{m, p}(0,T; X)$ be the space of functions in $W^{m, p}(0,T)$ with respect to $\|\cdot\|_X$.  Similarly, for $0\leq t_0 <T$, let $C([t_0,T];X)$ be the space of functions in $C([t_0,T])$ with respect to $\|\cdot\|_X$. Moreover, define $C((0,T];X):=\cap_{0<t_0<T} C([t_0,T];X)$. All spaces are equipped with standard norms \cite{AdaFou,Eva}.


For $0<\mu<1$, define the Riemann-Liouville fractional left-sided and right-sided integral operators
\begin{eqnarray*}
\begin{cases}
\ds{J}_{0+}^{\mu} z(t):=\frac{1}{\Gamma(\mu)} \int_0^t \frac{z(s)}{(t-s)^{1-\mu}} d s,\quad 0<t\leq T,\\[0.2in]
\ds{J}_{T-}^{\mu} z(t):=\frac{1}{\Gamma(\mu)} \int_t^T \frac{z(s)}{(s-t)^{1-\mu}} d s,\quad 0\leq t < T,
\end{cases}
\end{eqnarray*}
the left-sided and right-sided Riemann-Liouville fractional derivative operators 
\begin{equation*}
{D}_{0+}^{\mu}z(t):=\frac{d}{dt}({J}_{0+}^{1-\mu} z)(t), \quad {D}_{T-}^{\mu}z(t):=-\frac{d}{dt}({J}_{T-}^{1-\mu} z)(t),
\end{equation*}
respectively, and the left-sided Caputo fractional derivative operator $^c\p_t^\mu z(t):=({J}_{0+}^{1-\mu} z')(t)$ \cite{Jinbook}.  


In this work, we consider the smooth variable exponent $\alpha:(0,T) \to (0,1)$ which satisfies the following assumptions:\\
{\bf H1}. $0<\alpha(t)<1$ on $t\in [0,T]$ such that $\alpha(t)$ is three times differentiable with $|\alpha'(t)|+|\alpha''(t)|+|\alpha'''(t)|\leq L$ for some $L>0$. \\
{\bf H2}. The fractional order $\alpha(\cdot):(0,T)\to (0,1)$ can be analytically extended to the sector $\mathcal S=\{|\arg z|<\theta\}$ with $\theta\in(0,\frac\pi2)$ and admits the estimate
$$
|\alpha(z)| \le \alpha(0)<1,\quad z\in\mathcal{S}.
$$

Furthermore, we denote $\alpha_0=\alpha(0)$ for simplicity. Throughout the paper, we shall use $C$, with or without subscript, to denote a positive generic constant, and it may take a different value at each event. 

%

\subsection{A transformed model and its well-posdeness}
We follow the idea of the perturbation method in \cite{Zhe} to transfer the form of \eqref{Prob-IBVP}. The main idea of this method is to replace the variable-exponent kernel by a suitable kernel such that their difference serves as a low-order perturbation term, without affecting other terms in the model.

\begin{lemma}\label{lem-ibvp}
The model \eqref{Prob-IBVP} is equivalent to the following form:
\begin{equation}\label{eq-ibvp}
\begin{cases}
^c \partial_t^{\alpha_0} u - \Delta u + \tilde g' * u = \tilde g u_0+F, & (x,t)\in\Omega_T,\\
u(x,0) = u_0(x),& x\in\Omega,\\
u(x,t) = 0, & (x,t)\in\partial\Omega\times(0,T),
\end{cases}
\end{equation}
where the function $\tilde g$ is defined by
\begin{align*}
\tilde g(t) = \int_0^t \frac{t^{-\alpha(r)}}{\Gamma(1-\alpha(r))} (-\alpha'(r) \ln t + \gamma(1-\alpha(r)) \alpha'(r))dr, \quad \gamma(r)=\frac{\Gamma'(r)}{\Gamma(r)}.
\end{align*}
\end{lemma}
\begin{proof}
Firstly, we split the kernel $k(t)$ into two parts as follows
\begin{align}
k(t)&=\frac{t^{-\alpha(t)}}{\Gamma(1-\alpha(t))}=\frac{t^{-\alpha_0}}{\Gamma(1-\alpha_0)}+\int_0^t\p_z\frac{t^{-\alpha(z)}}{\Gamma(1-\alpha(z))}dz=:\beta_{1-\alpha_0}+\tilde g(t),\label{hh3}
\end{align}
where, by direct calculations and the notation $\gamma(z):=\frac{\Gamma'(z)}{\Gamma(z)}$,
\begin{align*}
\tilde g(t)=\int_0^t\frac{t^{-\alpha(z)}}{\Gamma(1-\alpha(z))}\big(-\alpha'(z)\ln t+\gamma(1-\alpha(z))\alpha'(z) \big)dz.
\end{align*}
As 
$$
\begin{aligned}
t^{-\alpha(z)}=&t^{-\alpha_0}t^{\alpha_0-\alpha(z)} =t^{-\alpha_0}e^{(\alpha_0-\alpha(z))\ln t}
\\
\leq& t^{-\alpha_0}e^{Cz|\ln t|}\leq t^{-\alpha_0}e^{Ct|\ln t|}\leq Ct^{-\alpha_0},\quad t\in(0,T],
\end{aligned}
$$
the function $\tilde g:(0,T]\to \mathbb R$ could be bounded as
\begin{align*}
|\tilde g|\leq C\int_0^t t^{-\alpha_0}(1+|\ln t|)dz\leq Ct^{1-\alpha_0}(1+|\ln t|)\rightarrow 0\text{ as }t\rightarrow 0^+. 
\end{align*}
By similar estimates we bound $\tilde g'$ as
\begin{align*}
|\tilde g'|\leq Ct^{-\alpha_0}(1+|\ln t|)\in L^1(0,T). 
\end{align*}
We then invoke (\ref{hh3}) in the first equation of \eqref{Prob-IBVP} to get
\begin{align}\label{qq1}
^c\p_t^{\alpha_0}u+\tilde g*\p_tu -\Delta u=F. 
\end{align}
Notice the zero initial condition, we apply the integration by parts to get
$$\tilde g*\p_tu=-\tilde gu_0+\tilde g'*u, $$
which, together with (\ref{qq1}), leads to
\begin{align*}
^c\p_t^{\alpha_0}u -\Delta u+\tilde g'*u=F+\tilde gu_0. 
\end{align*}
The proof of the lemma is complete.
\end{proof}

\begin{lemma}\label{lem-esti-g}
Assuming that the fractional order $\alpha(\cdot):(0,T]\to (0,1)$ satisfies the assumptions {\bf H1} and {\bf H2}, the function $\tilde g(t)$ can also be analytically extended to the sector $\mathcal S_\theta:=\{z\in\mathbb C;|\arg z|<\theta\}$ with $\theta\in(0,\frac\pi2)$. For any $K>0$, there exists a constant $C=C(\alpha,\theta,K)>0$ such that
$$
|\tilde g(z)| \le C|z|^{1-\alpha_0}(|\ln z|+1)
$$
and
$$
|\tilde g'(z)| \le C|z|^{-\alpha_0}(|\ln z|+1)
$$
for any $z\in \mathcal S_\theta\cap \{|z|<K\}$
\end{lemma}
\begin{proof} 
After the change of the variable $r/t \to r$ of the integral in the function $\tilde g$ defined in Lemma \ref{lem-ibvp}, we see that
\begin{align*}
\tilde g(t) = t\int_0^1 \frac{e^{-\alpha(rt)\ln t}}{\Gamma(1-\alpha(rt))} \big(\gamma(1-\alpha(rt)) \alpha'(rt)-\alpha'(rt)\ln t \big) dr.
\end{align*}
From the assumption {\bf H2}, the function $\tilde g(t)$ can be analytically extended to the right half plane $\{z\in\mathbb C; \Re z>0\}$, that is, 
\begin{align*}
\tilde g(z) = z\int_0^1 \frac{e^{-\alpha(rz)\ln z}}{\Gamma(1-\alpha(rz))} \big(\gamma(1-\alpha(rz)) \alpha'(rz)-\alpha'(rz)\ln z \big) dr.
\end{align*}
Here $\ln z$ is taken on its principal value branch, that is, $\ln z=\ln |z| + i \arg z$ with $\arg z\in (-\pi,\pi)$. By a direct calculation, we have
\begin{align*}
|\tilde g(z)| \le |z| \int_0^1 \frac{e^{|\alpha(rz)\ln z|}}{|\Gamma(1-\alpha(rz))|} \big(|\gamma(1-\alpha(rz)) \alpha'(rz)| + |\alpha'(rz)\ln z| \big) dr.
\end{align*}
We note that the functions $\frac{1}{\Gamma(z)}$, $\alpha(z)$ and $\gamma(z)$ are all analytic functions on complex plane, which implies that they are bounded on the compact set $\mathcal S_{\theta,K}$. Consequently, we have
\begin{align*}
|\tilde g(z)| \le C|z| \int_0^1 e^{|\ln z|} \big(1 + |\ln z| \big) dr,\quad z\in \mathcal S_{\theta,K},
\end{align*}
which combined with the fact $|\ln z| = \left((\ln |z|)^2 + (\arg z)^2 \right)^{\frac12} \le |\ln |z|| + \theta$ implies that
\begin{align*}
|\tilde g(z)| \le& Ce^{\theta} |z|e^{\alpha_0|\ln |z||}\int_0^1 \big(1 + |\ln z| \big) dr\\
\le& Ce^{\theta}\max\{|z|^{1-\alpha_0},|z|^{1+\alpha_0}\}\int_0^1 \big(1 + |\ln z| \big) dr
\le C|z|^{1-\alpha_0}(|\ln z|+1)
\end{align*}
for $z\in \mathcal S_{\theta,K}$. Next, we will give an estimate for the first order derivative $\tilde g'(z)$. For this, by a direct calculation, we see that
\begin{align*}
\tilde g'(z) =& \int_0^1 \frac{e^{-\alpha(rz)\ln z}}{\Gamma(1-\alpha(rz))} \big(\gamma(1-\alpha(rz)) \alpha'(rz)-\alpha'(rz)\ln z \big) dr\\
&+ z\int_0^1 \frac{\partial }{\partial z} \left(\frac{\gamma(1-\alpha(rz)) \alpha'(rz)}{\Gamma(1-\alpha(rz))}  \right) e^{-\alpha(rz)\ln z} dr  \\
&-  \int_0^1 \frac{\gamma(1-\alpha(rz)) \alpha'(rz)}{\Gamma(1-\alpha(rz))} e^{-\alpha(rz)\ln z} \left( \alpha'(rz)rz\ln z + \alpha(rz) \right) dr
\\
&- z\ln z\int_0^1 \frac{\partial }{\partial z} \left(\frac{ \alpha'(rz)}{\Gamma(1-\alpha(rz))}  \right) e^{-\alpha(rz)\ln z} dr\\
&- \int_0^1 \frac{\alpha'(rz)}{\Gamma(1-\alpha(rz))}  e^{-\alpha(rz)\ln z}  dr\\
&+ \ln z\int_0^1 \frac{\alpha'(rz)}{\Gamma(1-\alpha(rz))}  e^{-\alpha(rz)\ln z} \left( \alpha'(rz)rz\ln z + \alpha(rz) \right)   dr.
\end{align*}
By argument similar to estimating $\tilde g$, we arrive at the following inequalities for  $\tilde g'$:
\begin{align*}
|\tilde g'(z)| \le & C\int_0^1 e^{\alpha_0|\ln z|} \big(1+|\ln z| \big) dr + C|z|\int_0^1  e^{\alpha_0|\ln z|} dr  \\
&+ C \int_0^1  e^{\alpha_0|\ln z|} \left(1+ |z\ln z| \right) dr+
C|z||\ln z|\int_0^1  e^{\alpha_0|\ln z|} dr
+ C \int_0^1 e^{\alpha_0|\ln z|}  dr\\
&+ |\ln z|\int_0^1  e^{\alpha_0|\ln z|} \left(1+|z\ln z| \right)   dr
\\
\le & C|z|^{-\alpha_0}(|\ln z|+1) + C|z|^{1-\alpha_0} (1+ |\ln z| + |\ln z|^2),\quad z\in \mathcal S_{\theta,K}.
\end{align*}
Finally, by noting the estimate 
$$
|z|^{1-\alpha_0} (1+ |\ln z| + |\ln z|^2) \le C|z|^{-\alpha_0}(|\ln z|+1),\quad z\in \mathcal S_{\theta,K},
$$
it follows that
\begin{align*}
|\tilde g'(z)| \le &C|z|^{-\alpha_0}(|\ln z|+1) ,\quad z\in \mathcal S_{\theta,K}.
\end{align*}
This completes the proof of the lemma.
\end{proof}
\begin{corollary}
Assuming $\phi(z)$ is analytic function in the sector $\mathcal S$ satisfying $|\phi(z)|\le C|z|^{\rho}$ with $\rho\ge0$, then the convolution $g'*\phi(t), t\in(0,T)$ can be analytically extended to the sector $\mathcal S_\theta:=\{z\in\mathbb C;|\arg z|<\theta\}$, and the extended convolution admits the following estimate
$$
   |\tilde g'*\phi(z)| \le C|z|^{\rho},\quad z\in \mathcal S_\theta\cap \{|z|\le K\},
$$
where $\beta>\alpha_0$ and the constant $C>0$ only depends on $T,\rho,K,\alpha$.
\end{corollary}
\begin{proof}
    Firstly, noting from the definition of the convolution, we change the variable to get
$$
\tilde g'*\phi (t) = \int_0^t \tilde g'(r) \phi(t-r)dr = t\int_0^1 \tilde g'(rt) \phi(t-rt)dr
$$
We then directly extended to the complex function
$$
\tilde g'*\phi (z) := z\int_0^1 \tilde g'(rz) \phi(z-rz)dr,\quad z\in\mathcal S.
$$
Since the above Lemma \ref{lem-esti-g} implies that the integrand $\tilde g'(rz)\phi(z-rz)$ is integrable, we see that the extended convolution $\tilde g'*\phi(z)$ is actually an analytical function in the sector $\mathcal S$. Moreover, in view of the fact that $|z|^{-\alpha_0} |\ln z|\le C|z|^{-\beta}$ with $1>\beta>\alpha_0$ for $z\in \mathcal S_{\theta,K}$, we have
$$
\begin{aligned}
|\tilde g'*\phi (z)| \le& C|z| \int_0^1 |rz|^{-\alpha_0}(1+|\ln rz|) |\phi(z-rz)|dr
\\ 
\le& C|z|^{1+\rho-\beta} \int_0^1 r^{-\beta}(1-r)^{\rho}dr \le C|z|^{1+\rho-\beta}.
\end{aligned}
$$
This completes the proof of the corollary by using the inequality $|z|^{1+\rho-\beta} \le C|z|^\rho$ for $z\in \mathcal S_{\theta,K}$.
\end{proof}

Now we give the well-posedness of the problem (\ref{eq-ibvp}). 
%


\begin{lemma}\label{wellposedness-ivp}
Suppose $F\in W^{1,1}(0,T;L^2(\Omega))$ and $u_0\in H^2(\Omega)\cap H_0^1(\Omega)$, then the initial boundary value problem \eqref{Prob-IBVP} admits a unique solution in $W^{1,p}(0,T;L^2(\Omega))\cap L^p(0,T;H^2(\Omega)\cap H_0^1(\Omega))$ for $1<p<\frac{1}{1-\alpha_0}$ and 
\begin{equation*}
\|u\|_{W^{1,p}(0,T;L^2(\Omega))}+\|u\|_{L^p(0,T;H^2(\Omega))}\leq C\left(\|F\|_{W^{1,1}(0,T;L^2(\Omega))}+\|u_0\|_{H^2(\Omega)}\right).
\end{equation*}
Moreover, if $\alpha(t)$ can be analytically extended to a sector $\mathcal S_\theta:=\{z\in\mathbb C; |\arg z|<\theta\}$ with $\theta\in(0,\frac{\pi}{2})$, then the solution $u:(0,T)\to L^2(\Omega)$ to problem \eqref{Prob-IBVP} with $F\equiv0$ can be analytically extended to $(0,\infty)$.
\end{lemma}
\begin{proof}
The unique existence and the stability estimate can be found in \cite{Zhe}. So it suffices to show the $t$-analyticity. For this, by regarding the convolution term in \eqref{eq-ibvp} as a new source, from the argument used in Li et al. \cite{LiImYa}, we can rewrite the equivalent integral form of \eqref{eq-ibvp} as follows
\begin{align*}
u(t)=S_1(t)u_0 + t\int_0^1 \tilde g(rt) S_2(t-rt) u_0 d r
-t\int_0^1 S_2(t-rt) (\tilde g'*u)(rt) d r.
\end{align*}
Here $S_1(z),S_2(z):L^2(\Omega)\rightarrow L^2(\Omega)$ for $\Re z>0$ is defined by
\begin{align*}
S_1(z)a:=&\sum_{n=1}^\infty (a,\varphi_n)_{L^2(\Omega)} E_{\alpha_0,1}(-\lambda_n z^{\alpha_0})
\varphi_n, \\
S_2(z)a:=&-\sum_{n=1}^\infty (a,\varphi_n)_{L^2(\Omega)} z^{\alpha_0-1}E_{\alpha_0,\alpha_0}(-\lambda_n z^{\alpha_0}) \varphi_n,\quad a\in L^2(\Omega).
\end{align*}
Setting $u^0=0$, we inductively define $u^{n+1}(z)$ $(n=0,1,\cdots)$ as
follows:
\begin{align*}
u^{n+1}(z)
=S_1(z)u_0 + z\int_0^1 \tilde g(rz) S_2(z-rz) u_0 d r
- z\int_0^1 S_2(z-rz) (\tilde g'*u^n)(rz) d r.
\end{align*}
It is not difficult to check that $u^n (z)$ is analytic for any $n\in\mathbb N$. It remains to show that $u^n$ is uniformly convergent on any given compact set $\mathcal S_{\theta,K}$ with $K>0$. in view of the estimate (7) in \cite{LiImYa} for the operators $S_1$ and $S_2$, we see
\begin{align*}
\|u^{1}(z) - u^0(z)\|_{L^2(\Omega)}
\le&\left\|S_1(z)u_0\right\|_{L^2(\Omega)} + \left\|z\int_0^1 g(rz) S_2(z-rz) u_0 d r\right\|_{L^2(\Omega)}\\
\leq& C\|u_0\|_{L^2(\Omega)}
+ C|z|^{\alpha_0}\|u_0\|_{L^2(\Omega)} \int_0^1 (1-r)^{\alpha_0-1} d r\\
\leq& C\|u_0\|_{L^2(\Omega)}=:M_1.
\end{align*}
Next by induction we will prove that
\begin{align}
\label{esti-un}
\|u^{n+1}(z) - u^n(z)\|_{L^2(\Omega)}
\le \frac{C^{n}(| z|^{\alpha_0}\Gamma(\alpha_0))^n}{\Gamma(n\alpha_0+1)}M_1,
\quad n=0,1,2,..., \ \forall z \in \mathcal S_{\theta,K}.
\end{align}
Again by the use of the estimates (7) in \cite{LiImYa}, and noting the estimates in Lemma \ref{lem-esti-g}, it follows that
\begin{align*}
\|u^{n+1}(z) - u^n(z)\|_{L^2(\Omega)}
\leq& C\left\| z\int_0^1 S_2(z-rz) (\tilde g'*(u^n-u^{n-1})(rz) d r \right\|_{L^2(\Omega)}\\
\leq& C|z|^{\alpha_0} \int_0^1 (1-r)^{\alpha_0-1} \|\tilde g'*(u^n-u^{n-1})(rz)\|_{L^2(\Omega)} d r .
\end{align*}
By the assumption (\ref{esti-un}) of induction, we derive that
\begin{align*}
\|u^{n+1}(z) - u^n(z)\|_{L^2(\Omega)}
\leq& C  |z|^{\alpha_0}
\int_0^1 (1-r)^{\alpha_0-1} \frac{C^{n-1}(| z|^{\alpha_0}\Gamma(\alpha_0))^{n-1}}{\Gamma((n-1)\alpha_0+1)}M_1 d r \nonumber\\
=& C|z|^{\alpha_0}M_1\frac{(C|z|^{\alpha_0}\Gamma(\alpha_0))^{n-1}}{\Gamma((n-1)\alpha_0+1)} B(\alpha_0, (n-1)\alpha_0+1).
\end{align*}
Noting the relation between Gamma function and Beta function
$$
B(\alpha_0,(n-1)\alpha_0+1) = \frac{\Gamma(\alpha_0) \Gamma((n-1)\alpha_0+1)}{\Gamma(n\alpha_0+1)},
$$
we finally obtain that
\begin{align*}
\|u^{n+1}(z) - u^n(z)\|_{L^2(\Omega)}
\leq&  M_1\frac{(C|z|^{\alpha_0}\Gamma(\alpha_0))^{n}}{\Gamma(n\alpha_0+1)} .
\end{align*}
Hence the proof of \eqref{esti-un} is completed by induction.  Therefore, for any $z\in\mathcal S_{\theta,K}$, we have
\begin{equation*}
\Vert u^n(z)\Vert_{L^2(\Omega)}
\le \sum_{k=0}^{n-1} \Vert u^{k+1}(z) - u^k(z)\Vert_{L^2(\Omega)}
\le \sum_{k=0}^{\infty} M_1\frac{(CK^{\alpha_0}\Gamma(\alpha_0))^n}
{\Gamma(n\alpha_0+1)}< \infty,
\end{equation*}
where the last inequality is due to the asymptotic expansion of the Gamma function
\begin{equation*}
\Gamma(\eta) = e^{-\eta}\eta^{\eta-\frac 1 2}(2\pi)^{\frac 1 2}
\left(1 + O\left(\frac 1\eta \right)\right),\mbox{ as }\eta\rightarrow +\infty
\end{equation*}
(e.g., Abramowitz and Stegun \cite{AS}, p.257). Therefore, there exists $u(z)\in L^2(\Omega)$ such that
$\|u^n(z) - u(z)\|_{L^2(\Omega)}$ tends to $0$ as $n \rightarrow \infty$ uniformly in $z\in \mathcal S_{\theta,K}$. Recalling the analyticity of $u^n$ in $z\in \mathcal S_{\theta,K}$ for $n=1,2,\cdots$, we see that $z\mapsto u(z)\in L^2(\Omega)$ is analytic in $\mathcal S_{\theta,K}$. Moreover, since $K$ is arbitrarily chosen, we deduce $u(z)$ is analytic in the sector $\mathcal S_\theta$. We complete the proof of the theorem.
\end{proof}

\subsection{Weak unique continuation principle}
In this section, we consider the unique continuation principle of the solution of the homogeneous variable-exponent sub-diffusion equation with $F\equiv0$,
\begin{equation}\label{eq-F=0}
\begin{cases}
^c \partial_t^{\alpha_0} u - \Delta u + \tilde g' * u = \tilde g u_0, & (x,t)\in\Omega_T,\\
u(x,0) = u_0(x),& x\in\Omega,\\
u(x,t) = 0, & (x,t)\in\partial\Omega\times(0,T).
\end{cases}
\end{equation}
The distinctive feature of solution propagation in fractional diffusion equations is their unique continuation property, which is widely utilized to address inverse source problems and to assess the feasibility of approximate controllability, as explored by Fujishiro and Yamamoto \cite{FuYa}, and further elaborated by Jiang, Li, Liu, and Yamamoto \cite{Jiang17}.

Since the solution to the problem (\ref{eq-F=0}) is $t$-analytic, we can analytically extend the domain $(0,T]$ to the interval $(0,\infty)$. Therefore, the Laplace transform could be used to the above system and we get
$$
\begin{cases}
s^{\alpha_0} \hat u(s) - \Delta \hat u(s) + \widehat{\tilde g'}  \hat u(s) = \widehat{\tilde g} u_0 + s^{\alpha_0-1} u_0, & x\in\Omega,\\
\hat u(s) = 0, & x\in\partial\Omega.
\end{cases}
$$
In view of the fact that $\widehat{\tilde g'} = s\widehat{\tilde g} + \tilde g(0) = s\widehat{\tilde g}$, the above system in the frequency domain can be rephrased as follows
$$
\begin{cases}
s^{\alpha_0} \hat u(s) - \Delta \hat u(s) + [s^{\alpha_0} + s\widehat{\tilde g}(s)]  \hat u(s) = [s^{\alpha_0-1}+\widehat{\tilde g}(s) u_0, & x\in\Omega,\\
\hat u(s) = 0, & x\in\partial\Omega.
\end{cases}
$$

By $\left\{ \lambda_n, \varphi_n \right\}_{n=1}^{\infty}$ the Dirichlet eigensystem of the operator $-\Delta$, we conclude from the Fourier series theory that 
$$
\hat u(x,s) = \sum_{n=1}^\infty (u_0,\varphi_n) \varphi_n \frac{s^{\alpha_0-1} + \widehat{\tilde g}(s)}{s^{\alpha_0} + s\widehat{\tilde g}(s) + \lambda_n}.
$$

The denominator in the above expression is not zero in some open set in complex plane. In fact, we have the following lemma.

\begin{lemma}
There holds that
$$
s^{\alpha_0-1} + \widehat{\tilde g}(s) \not\equiv0.
$$
\end{lemma}
\begin{proof}
If it is not true, that is, $\widehat{\tilde g} = -s^{\alpha_0-1}$ for any $\Re s\ge s_0$, where $s_0>0$ is a suitable choice constant. In this case, from the inversion Laplace transform formula it follows that
$$
\tilde g(t) = -\frac{t^{-\alpha_0}}{\Gamma(1-\alpha_0)},
$$
that is
$$
\int_0^t \frac{t^{-\alpha(z)}}{\Gamma(1-\alpha(z))} (\gamma(1-\alpha(z)) \alpha'(z) -\alpha(z) \ln t) dz = -\frac{t^{-\alpha_0}}{\Gamma(1-\alpha_0)}.
$$
It is not difficult to check that $\tilde g(t) \to 0$ ($t\to0$), but the term on the right-hand side $\frac{t^{-\alpha_0}}{\Gamma(1-\alpha_0)} \to \infty$ as $t$ tends to $0$. We get a contradiction. So we must have $s^{\alpha_0-1} + \widehat{\tilde g}(s) \not\equiv0$, which completes the proof of the lemma.
\end{proof}

\begin{theorem}\label{continuation-interior}
Suppose $u$ is a solution to system \eqref{eq-ibvp} satisfying $u(x,t)=0$ for $(x,t)\in\Omega_0\times(0,T)$, where $\Omega_0$ is a subdomain of $\Omega$, then $u(x,t)=0$ in $\Omega\times(0,T)$.
\end{theorem}
\begin{proof}
From the $t$-analyticity of the solution $u$ to the system \eqref{eq-ibvp}, we see that $\hat u(x,s)=0$, $x\in\Omega$. Consequently,
$$
\sum_{n=1}^\infty (u_0,\varphi_n) \varphi_n \frac{ s^{\alpha_0-1} + \widehat{\tilde g}(s)}{\lambda_n + s^{\alpha_0} + s\widehat{\tilde g}(s)} = 0,\quad x\in\Omega_0.
$$
Letting $\eta =  s^{\alpha_0} + s\widehat{\tilde g}(s)$, then we get
$$
\sum_{n=1}^\infty (u_0,\varphi_n) \varphi_n \frac{1}{\lambda_n + \eta} = 0,\quad \eta\in \mathcal O,
$$
where $\mathcal O\subset \mathbb C$ is an open set. By using the analytic continuation, we know that the above equation is actually valid for any $\eta\neq -\lambda_n$, that is,
$$
\sum_{n=1}^\infty (u_0,\varphi_n) \varphi_n \frac{1}{\lambda_n + \eta} = 0,\quad x\in\Omega_0.\ \eta\neq -\lambda_n,\ n=1,2,\cdots.
$$
We can follow the argument used in Sakamoto and Yamamoto \cite{SakYam} to get $u_0=0$ in $\Omega$. Finally, from the uniqueness of the initial boundary value problem \eqref{eq-ibvp}, we can obtain $u=0$ in $\Omega\times(0,T)$. The proof is finished.
\end{proof}
We choose any nonempty and open subboundary $\Gamma\subset \partial\Omega$, and follow the above argument used to prove Theorem \ref{continuation-interior}, we can also show the uniqueness of the Cauchy problem in the weak sense.
\begin{theorem}\label{continuation-boundary}
Suppose $u$ is a solution to system \eqref{eq-ibvp} with $F=0$, moreover, we assume $\partial_\nu u(x,t)=0$ for $(x,t)\in\Gamma\times(0,T)$, where $\Gamma$ is an open and nonempty subboundary of $\partial\Omega$, then $u(x,t)=0$ in $\Omega\times(0,T)$.
\end{theorem}

\section{\large Uniqueness and conditional stability for inverse source problem}\label{sec3}
In this section, we consider 
\begin{equation}\label{eq-isp}
\begin{cases}
^c \partial_t^{\alpha_0} u - \Delta u + \tilde g' * u = f(x)\beta(t), & (x,t)\in\Omega_T,\\
u(x,0) = 0, & x\in\Omega,\\
u(x,t) = 0, & (x,t)\in\partial\Omega\times(0,T).
\end{cases}
\end{equation}
Suppose that the temporal component $\beta$ is known, while $f$ is unknown to be recovered. The inverse problems are to identify the unknown $f$ in (\ref{eq-isp}) from the partial interior observation $u|_{\Omega_0\times(0,T)},\,\Omega_0\subset\Omega$ or the partial boundary observation $\partial_\nu u|_{\Gamma\times(0,T)},\,\Gamma\subset\partial\Omega$, respectively. In what follows, we will state the uniqueness and conditional stability of these two inverse source problems.

\subsection{Uniqueness of solution}
We need the following Duhamel's principle.
\begin{lemma}
The solution $u$ to the system \eqref{eq-isp} with $\beta\in W^{1,1}(0,T)$ can be represented as follows
$$
u(t) = \int_0^t \theta(t-\tau) v(\tau) d\tau=\theta*v(t),
$$
where $\theta$ is such that $J^{1-\alpha_0} \theta = \beta$, and $v$ is the solution to the initial-boundary value problem 
\begin{equation*}
\begin{cases}
^c\partial_t^{\alpha_0} v - \Delta v + \tilde g' * v = \tilde g f, & (x,t)\in\Omega_T,\\
v(x,0) = f(x),& x\in\Omega,\\
v(x,t) = 0, & (x,t)\in\partial\Omega\times(0,T).
\end{cases}
\end{equation*}
\end{lemma}
\begin{proof}
The initial and boundary conditions are obviously met. We only need to check that the right-hand side of the representation formula $\theta*v$ satisfies the initial boundary value problem \eqref{eq-ibvp} in view of the uniqueness of the forward problem.
$$
^c \partial_t^{\alpha_0}u  = J^{1-\alpha_0} \partial_t \int_0^t \theta(\tau) v(t-\tau) d\tau.
$$
A direct calculation yields
$$
^c \partial_t^{\alpha_0} = J^{1-\alpha_0} \theta(t) v(0) + \int_0^t \theta(\tau) {^c \partial_t^{\alpha_0}v}(t-\tau) d\tau.
$$
On the other hand, from the properties of the convolution we see that
$$
\tilde g'*u = \tilde g'*(\theta*v) = \theta*(\tilde g'*v) = \int_0^t \theta(\tau) \tilde g'* v(t-\tau) d\tau.
$$
Collecting all the above results, it follows that
$$
\begin{aligned}
&^c\partial_t^{\alpha_0} u - \Delta u + \tilde g'*u 
\\
=& J^{1-\alpha_0}\theta(t) v(0) + \int_0^t \theta(\tau)[{^c \partial_t^{\alpha_0}v - \Delta v + \tilde g'*v}] (t-\tau ) d\tau =\beta(t) f(x).
\end{aligned}
$$
We finish the proof of the lemma.
\end{proof}

\begin{theorem}\label{uniqueness-interior}
Suppose $\beta\neq 0$, $\beta\in W^{1,1}(0,T)$, $u$ is a solution to the system \eqref{eq-isp} satisfying $u(x,t)=0$ for $(x,t)\in\Omega_0\times(0,T)$, where $\Omega_0$ is a subdomain of $\Omega$, then $f(x)=0$ in $\Omega$.
\end{theorem}
\begin{proof}
From the Duhamel principle in the above lemma, it follows that
$$
u(t) = \int_0^t \theta(t-\tau) v(\tau) d\tau=\theta*v(t),
$$
moreover, by using the observation data $u(x,t)=0$, $(x,t)\in \Omega_0\times(0,T)$, we have
$$
0 = \int_0^t \theta(t-\tau) v(\tau) d\tau,\quad (x,t)\in\Omega_0\times(0,T).
$$
Multiplying $J^{1-\alpha_0}$ on both sides of the above equation, and noting $J^{1-\alpha_0}\theta = \beta$ derives that
$$
0=\int_0^t \beta(t-\tau) v(\tau) d\tau = 0,\quad (x,t)\in\Omega_0\times(0,T).
$$
Since $\beta\neq0$, we see from Titchmarsh convolution theorem that $v$ must be vanished in $\Omega_0\times(0,T)$. Therefore, noting the weak unique continuation principle in Theorem \ref{continuation-interior}, we finally see that $v\equiv0$, which implies $f=v(\cdot,0)=0$. We finish the proof of the lemma.
\end{proof}
By making minor modifications to the proof of Theorem \ref{uniqueness-interior} and incorporating Theorem \ref{continuation-boundary}, we can demonstrate the uniqueness of the inversion source term by using the flux data $\partial_\nu u$ on the subboundary $\Gamma\subset \partial\Omega$.
\begin{theorem}\label{thm-unique'}
Suppose $\beta\neq 0$, $\beta\in W^{1,1}(0,T)$, and let $u$ be a solution to the system \eqref{eq-isp} satisfying $\partial_\nu u(x,t)=0$ for $(x,t)\in\Gamma\times(0,T)$, where $\Gamma$ is a nonempty open subboundary of $\partial\Omega$, then $f(x)=0$ in $\Omega$.
\end{theorem}
\subsection{Conditional stability of Lipschitz type by partial interior observation}

Firstly, we introduce the adjoint system of \eqref{eq-isp} as
\begin{eqnarray}\label{prob-adjoint}
\begin{cases}
D_{T-}^{\alpha_0}{\phi} - \Delta \phi + \tilde g' \star \phi=\chi_{\Omega_0}\omega(x,t), &(x,t)\in \Omega_T, \\
\phi(x,t)=0, &(x,t)\in \partial\Omega\times[0,T],\\
J_{T-}^{1-\alpha_0}\phi(x,T) = 0, &x\in \Omega.
\end{cases}
\end{eqnarray}
Here $\tilde g' \star \phi:= \int_t^T \tilde g'(\tau-t) \phi(\tau) d\tau$, $\chi_{\Omega_0}$ denotes the characterization function of $\Omega_0$. Similarly as in Lemma \ref{wellposedness-ivp}, for $\omega\in W^{1,1}(0,T;L^2(\Omega_0))$, one can prove that the adjoint problem (\ref{prob-adjoint}) admits a unique solution in $W^{1,p}(0,T;L^2(\Omega))\cap L^p(0,T;H^2(\Omega)\cap H_0^1(\Omega))$ for $1<p<\frac{1}{1-\alpha_0}$. The main idea to prove that is to transfer the backward problem to a forward problem and employ the argument for the forward problem to prove the uniqueness and existence. First,  in view of the notations in Section 2.1, by changing the variable, we can rewrite the backward Riemann-Liouville fractional derivative as follows:
$$
D_{T-}^{\alpha_0} \phi(t) = \frac1{\Gamma(1-\alpha_0)} \frac{ \mathrm{d} }{ \mathrm{d} t}\int_0^{T-t} (T-t-\tau)^{-\alpha_0} \phi'(T-\tau) d\tau.
$$
Moreover, by setting $\phi_T(t):=\phi(T-t)$ and defining a new variable $\tilde t:=T-t$, we obtain that
$$
D_{T-}^{\alpha_0} \phi(t) = \frac1{\Gamma(1-\alpha_0)} \frac{ \mathrm{d} }{ \mathrm{d} \tilde t}\int_0^{\tilde t}  \left(\tilde t-\tau \right)^{1-\alpha_0} \phi_T'(\tau) d\tau = D_{\tilde t}^{\alpha_0} \phi_T(\tilde t).
$$
Similarly, under the above settings and treatment, we can show that $\widetilde g'*\phi(t)$ becomes
$$
\widetilde g'*\phi(t) = \widetilde g'* \phi_T(\tilde t).
$$
Now  the backward problem can be rewritten by
\begin{align*}
\begin{cases}
D_{t}^{\alpha_0}{\phi_T} - \Delta \phi_T + \tilde g' * \phi_T = \chi_{\Omega_0}\omega(x,T-t), &(x,t)\in \Omega_T, \\
\phi_T(x,t)=0, &(x,t)\in \partial\Omega\times[0,T],\\
\lim_{t\to0+}J^{1-\alpha_0}\phi_T(x,t) = 0, &x\in \Omega.
\end{cases}
\end{align*}
It is not difficult to check that $\chi_{\Omega_0}w(x,T-t)$ is in the space $W^{1,1}(0,T;L^2(\Omega))$, therefore by an argument similar to Lemma \ref{wellposedness-ivp} we can prove our above assertion. Then we can give the variational identity below, which will be applied to establish the Lipschitz stability by some weak norm to our inverse problem. Before that, we remark that there holds the integration by parts \cite{Jinbook}: 
\begin{lemma}\label{lem-igbp}
Let $\alpha_0>0$, $p\ge 1$, $q\ge 1$, and $1/p+1/q\leq 1+\alpha_0$ ($p\neq 1$ and $q\neq 1$ in the case where $1/p+1/q=1+\alpha_0$). If $u\in L^p(0,T)$ and $v\in L^q(0,T)$, then
\begin{eqnarray*}
\int\limits_0^T {u(t)}J_{0+}^{\alpha_0}v(t) dt = \int\limits_0^T {J_{T-}^{\alpha_0}u(t)}v(t) dt.
\end{eqnarray*}
\end{lemma}

Then, by the solution regularity for the problems \eqref{eq-isp} and \eqref{prob-adjoint} we have the following lemma.
\begin{lemma}\label{lem-var}
Suppose $\beta\in W^{1,1}(0,T)$. Let $u[f]$ be the solution to \eqref{eq-isp} with respect to the source function $f\in L^2(\Omega)$. There holds the identity
\begin{eqnarray}\label{var-identity}
\int_{\Omega_T}\beta(t)f(x)\phi[\omega](x,t) dxdt =
\int_{\Omega_0\times(0,T)}u[f](x,t)\omega(x,t) dx dt,
\end{eqnarray}
where $\phi[\omega](x,t)$ is the solution to \eqref{prob-adjoint} corresponding to $\omega\in W^{1,1}(0,T;L^2(\Omega_0))$.
\end{lemma}
\begin{proof}
Multiplying two sides of the first equation in \eqref{eq-isp} by $\phi(x,t):=\phi[\omega](x,t)$ and integrating in $\Omega_T$, we get
\begin{eqnarray*}
\int_{\Omega_T} (^c \partial_t^{\alpha_0} u - \Delta u + \tilde g' * u)(x,t)\phi(x,t) dxdt=
\int_{\Omega_T} \beta(t) f(x)\phi(x,t) dxdt.
\end{eqnarray*}
Under the given conditions, we know $u, \phi\in W^{1,p}(0,T;L^2(\Omega))\cap L^p(0,T;H^2(\Omega)\cap H_0^1(\Omega))$ for $1<p<\frac{1}{1-\alpha_0}$. Hence it is possible to apply Lemma \ref{lem-igbp} to have
\begin{equation*}
\begin{aligned}
\int_{\Omega_T}\ {^c\partial_t^{\alpha_0}}u \  \phi\  dx dt &=\int_{\Omega_T} J_{0+}^{1-\alpha_0} {\p_t u}\cdot \phi \ dx dt = \int_{\Omega_T} {\p_t u} \ J_{T-}^{1-\alpha_0}\phi \ dx dt \\
&={\int_{\Omega}  \big[u \,(J_{T-}^{1-\alpha_0}\phi)\big]|_{t=0}^{t=T}\; dx}-\int_{\Omega_T}  u \, ({\partial_t} J_{T-}^{1-\alpha_0}\phi) \ dx dt \\
&=\int_{\Omega_T}  u \, D_{T-}^{\alpha_0} \phi \ dx dt,
\end{aligned}
\end{equation*}
where we used $u(x,0)=0$ and $J_{T-}^{1-\alpha_0}\phi(x,T) = 0$. Next, combing the zero boundary conditions, we have from Green's formula that
\begin{eqnarray*}
\int_{\Omega_T} (- \Delta u + \tilde g' * u) \ \phi\ dxdt
=\int_{\Omega_T} (- \Delta \phi + \tilde g' \star \phi)\ u \ dxdt.
\end{eqnarray*}

Combing all the above results with the first equation in \eqref{prob-adjoint}, we complete the proof of this lemma.
\end{proof}

Let $\beta\in W^{1,1}(0,T)$, $\beta \not\equiv 0$ in $[0,T]$ and
$$\mathcal{W}:=\{\omega: \; \omega \in W^{1,\infty}(0,T;L^\infty(\Omega_0)), \,\omega\neq 0\}.$$ Define a bilinear functional with respect to $f\in L^2(\Omega)$ and $\omega\in\mathcal{W}$ in terms of $\beta$ by
\begin{eqnarray}\label{def-B}
\mathcal{B}(f,\omega):=\int_{\Omega_T}\beta(t)f(x)\phi[\omega](x,t) \, d x d t,
\end{eqnarray}
where $\phi[\omega](x,t)$ defined by \eqref{prob-adjoint} is a linear functional of $\omega\in \mathcal{W}$. It is easy to see that $\mathcal{B}(f,\omega)$ in (\ref{def-B}) is well defined for $f\in L^2(\Omega)$ and $\omega\in \mathcal{W}$. Define
\begin{eqnarray}\label{sun4-2}
\|{f}\|_\mathcal{B}:=\frac{1}{\|{\beta}\|_\infty}\mathop {\sup}\limits_{\omega\in\mathcal{W} }\frac{|\mathcal{B}(f,\omega)|}{\|{\omega}\|_{L^\infty(\Omega_0\times(0,T))}}.
\end{eqnarray}
As in \cite{sun20,Sun2020IP}, we can prove that $\|\cdot\|_\mathcal{B}$ is a norm on $L^2(\Omega)$ and $\|{f}\|_\mathcal{B} \leq C\|{f}\|_{L^2(\Omega)}$. The key for this is to prove that $\|f\|_{\mathcal B}=0$ leads to $f=0$, and this can be proven by the variational identity (\ref{var-identity}) together with the uniqueness of the inverse problem. Then we can establish the conditional stability of Lipschitz type for the inverse problem using the weaker norm $\|\cdot\|_\mathcal{B}$.

\begin{theorem}\label{thm-lip-B}
For the direct problem \eqref{eq-isp} with $\beta\in W^{1,1}(0,T)$ and $\beta \not\equiv 0$ in $[0,T]$, denote by $u[f_i](x,t)$ the solution to \eqref{eq-isp} corresponding to $f_i\in L^2(\Omega)$ for $i=1,2$. Then it follows that
\begin{eqnarray*}
\|{f_1-f_2}\|_{\mathcal{B}}\leq C\|{ u[f_1]-u[f_2]}\|_{L^1(\Omega_0\times(0,T))}.
\end{eqnarray*}
\end{theorem}
\begin{proof}
By the Lemma \ref{lem-var} and the definition (\ref{sun4-2}), we have
\begin{equation*}
\begin{aligned}
\|{f_1-f_2}\|_{\mathcal{B}} &= \frac{1}{\|\beta\|_{\infty} }\mathop {\sup}\limits_{\omega\in {\mathcal W} }\frac{|\mathcal{B}(f_1-f_2,\omega)|}{\|{\omega}\|_{L^\infty(\Omega_0\times(0,T))}} \nonumber\\
&=\frac{1}{\|\beta\|_{\infty} }\mathop {\sup}\limits_{\omega\in {\mathcal W} }\frac{\left|\int_0^T\int_{\Omega_0}{(u[f_1]-u[f_2]) \omega(x,t) {\rm d}x {\rm d}t}\right|}{\|{\omega}\|_{L^\infty(\Omega_0\times(0,T))}} \\
&\leq C\|{ u[f_1]-u[f_2]}\|_{L^1(\Omega_0\times(0,T))}.
\end{aligned}
\end{equation*}
The proof is complete.
\end{proof}
\subsection{Conditional stability of Lipschitz type by partial boundary data}
Now we consider the conditional stability of the inverse source problem by using partial boundary observation data. For this, we introduce the adjoint system of \eqref{eq-isp} as
\begin{eqnarray}\label{prob-adjoint'}
\begin{cases}
D_{T-}^{\alpha_0}{\tilde \phi} - \Delta \tilde\phi + \widetilde g' \star \tilde\phi= 0, &(x,t)\in \Omega_T, \\
\tilde\phi(x,t)= \chi_{\Gamma}\omega(x,t), &(x,t)\in \partial\Omega\times[0,T],\\
J_{T-}^{1-\alpha_0}\tilde\phi(x,T) = 0, &x\in \Omega.
\end{cases}
\end{eqnarray}
Here, the boundary input $\chi_\Gamma$ is the characteristic function of $\Gamma$ and $w\in W^{2,1}(0,T;L^2(\partial\Omega))$. We next consider the wellposedness of the above adjoint system with a non-homogeneous boundary condition. For this, we construct $\Phi$ satisfies the following elliptic problem
\begin{align*}
\begin{cases}
 -\Delta \Phi = 0 &\text{ in } \Omega_T, \\
\Phi(x,t) = \chi_{\Gamma}\omega(x,t) &\text{ on } \partial \Omega \times (0, T).
\end{cases}
\end{align*}
From the theories of elliptic equations, we can see that $\Phi\in W^{2,1} \left(0,T;H^2(\Omega) \right)$, and moreover, we can see that $W= \tilde\phi-\Phi$ satisfies
\begin{align*}
\begin{cases}
D_{T-}^{\alpha_0} W - \Delta  W = -D_{T-}^{\alpha_0} \Phi  - \widetilde g'\star \Phi  &\text{ in } \Omega_T, \\
J_{T-}^{1-\alpha_0}W(x,T) = 0, & \text{ in } \Omega,\\
W(x,t) = 0 &\text{ on } \partial \Omega \times (0, T).
\end{cases}
\end{align*}
The right-hand side is in the space $W^{1,1}(0,T;L^2(\Omega))$, then the adjoint system admits a weak solution $W$ that satisfies $W^{1,p}(0,T;L^2(\Omega))\cap L^p(0,T;H^2(\Omega)\cap H_0^1(\Omega))$ for $1<p<\frac{1}{1-\alpha_0}$. Consequently, we see $\tilde\phi\in W^{1,p}(0,T;L^2(\Omega))\cap L^p(0,T;H^2(\Omega)\cap H_0^1(\Omega))$.

\begin{lemma}\label{var-indentity'}
Let $u[f]$ be the solution to the problem \eqref{eq-isp}. Then we have
\begin{align*}
\int_0^T \int_\Gamma  \left( \partial_{\nu} u \right) \omega \,\mathrm{d}x \mathrm{d}t = \int_{\Omega_T}   f(x) g(t) \tilde\phi[\omega](x,t) dxdt,
\end{align*}
where $w\in W^{2,1}(0,T;L^2(\partial\Omega))$ and $\chi_\Gamma$ is the characteristic function of the partial boundary $\Gamma\subset\partial\Omega$, and $\tilde\phi[\omega]$ is the solution to the problem (\ref{prob-adjoint'}).
\end{lemma}
The proof is extremely similar to Lemma \ref{lem-var} using integration by parts and Lemma \ref{lem-igbp}, and we omit the proof. 

Now by using a similar treatment and setting in the above subsection, we let $\beta\in W^{1,1}(0,T)$, $\beta \not\equiv 0$ in $[0,T]$ and define
$$
\mathcal{M}:=\{\omega: \; \omega \in W^{2,\infty}(0,T;L^\infty(\Gamma)), \; \omega\neq0 \}.
$$
Define a bilinear functional with respect to $f\in L^2(\Omega)$ and $\omega\in\mathcal{M}$ in terms of $\beta$ by
\begin{eqnarray*}
{\mathcal{D}}(f,\omega):=\int_{\Omega_T}\beta(t)f(x)\tilde\phi[\omega](x,t) \, d x d t,
\end{eqnarray*}
where $\tilde\phi[\omega](x,t)$ defined by \eqref{prob-adjoint} is a linear functional of $\omega\in \mathcal{M}$. It is easy to see that $\mathcal{D}(f,\omega)$ is well defined for $f\in L^2(\Omega)$ and $\omega\in \mathcal{M}$. Define
\begin{eqnarray*}
\|{f}\|_{\mathcal{D}}:=\frac{1}{\|{\beta}\|_\infty}\mathop {\sup}\limits_{\omega\in\mathcal{M}}\frac{|\mathcal{D}(f,\omega)|}{\|{\omega}\|_{L^\infty(\Gamma\times(0,T))}}.
\end{eqnarray*}
Similar to Theorem \ref{thm-lip-B}, from the uniqueness of the inverse source problem in Theorem \ref{thm-unique'} and Lemma \ref{var-indentity'}, we can establish the conditional stability of Lipschitz type for the inverse problem under the norm $\|\cdot\|_\mathcal{D}$.

\begin{theorem}\label{thm-lip-D}
For the direct problem \eqref{eq-isp} with $\beta\in W^{1,1}(0,T)$ and $\beta \not\equiv 0$ in $[0,T]$, denote by $u[f_i](x,t)$ the solution to \eqref{eq-isp} corresponding to $f_i\in L^2(\Omega)$ for $i=1,2$. Then it follows that
\begin{eqnarray*}
\|{f_1-f_2}\|_{\mathcal{D}}\leq C\|\partial_\nu u[f_1]- \partial_\nu u[f_2]\|_{L^1(\Gamma\times(0,T))}.
\end{eqnarray*}
\end{theorem}

\section{\large Numerical inversion}\label{sec4}
In this section, we will investigate the numerical inversion of the unknown source from the partial interior observation. Denote by $g$ the exact measurement data. Let $g^\delta$ be the noisy data of $g$, satisfying $\|g^\delta-g\|_2 \leq \delta$. Then, the concerned inverse problems are to solve the following operator equation
\begin{equation}\label{operator-G}
\mathbb{G}f\approx g^\delta,
\end{equation}
where the observation operator $\mathbb{G}: L^2(\Omega)\to L^2(\Omega_0\times(0,T))$ is defined by $\mathbb{G}f:=u[f]|_{\Omega_0\times(0,T)}$. For the noisy data $g^\delta$, we describe it as $g^\delta = g(1+\delta\zeta)$, where $\delta\ge0$ is the noise level and $\zeta$ is a random standard Gaussian noise. Furthermore, due to the ill-posedness of inverse problems, we will use iterative regularization algorithms to solve stably (\ref{operator-G}). 

Let $f_n$ be the iterative source function in the $n$-th iteration step. We set the initial guess as $f_0=0$ and denote by
$f_{\rm true}$ and $f_{\rm inv}$ the true solution and the recovered source function, respectively.  For the numerical examples, we set $\alpha(t)=0.5+0.25t$. Suppose $T=1$ and $\Omega:=[0,1]^2$, we solve the forward problem \eqref{eq-isp} and the adjoint problems using the finite element method and give a uniform mesh with the size of the mesh time $0.01$. If not specified, we set the interior observation area $\Omega_0=[0.3,0.7]^2$.


We first do the numerical inversion for smooth sources. Although we already establish the stability of our inverse problem by the norm $\|\cdot\|_{\mathcal B}$ in Theorem \ref{thm-lip-B}, such a kind of norm as the penalty term is not convenient for numerical implementations. So here we still take the Tikhonov regularization with $L^2$ norm penalty to solve the approximate solution of \eqref{operator-G}, i.e., minimizing the following regularization functional
\begin{equation}\label{L^2-term}
\mathcal{J}(f):=\frac{1}{2}\|\mathbb{G}f-g^\delta\|_2^2+\frac{\epsilon}{2}\|f\|_2^2.
\end{equation}
It is easy to see that $\nabla {\mathcal J}(f)=(\mathbb{G}^*\mathbb{G}+\epsilon I)f-\mathbb{G}^* g^\delta$, where $\mathbb{G}^*$ is the adjoint operator of $\mathbb{G}$ and is defined by $\mathbb{G}^*\omega:=\int_0^T \beta(t)\phi[\omega](x,t)\,dt$. Here $\phi[\omega]$ is the solution to the adjoint problem \eqref{prob-adjoint}. Then the function $f^*\in L^2(\Omega)$ is a minimizer of the functional $\mathcal{J}(f)$ only if it satisfies the variational equation $\mathbb{G}^*(\mathbb{G}f^*-g^\delta)+\epsilon f^*=0$. Hence, we can apply the iterative thresholding algorithm for the reconstruction. The iteration process is given by
$$
f_{n+1}=\frac{1}{A+\epsilon}(Af_n-\mathbb{G}^*(\mathbb{G}f_n-g^\delta))
$$
and stops if $\|f_{n+1}-f_n\|_2/\|f_n\|_2 \leq \varrho$, where $A=30$ is a tuning parameter and $\varrho=2\times 10^{-4}$ is a tolerance parameter.

\begin{example}\label{ex1}
Consider the numerical inversion for following cases:
\begin{itemize}
\item[(1a):]  the smooth solution is given by
$$
f_{\rm true}(x,y)=1+\sin(\pi x)\sin(\pi y), \; (x,y)\in [0,1]^2;
$$
\item[(1b):] the smooth solution is given by
$$
f_{\rm true}(x,y)=1+\sin(\pi x)\sin(\pi y)+\sin(2\pi x)\sin(2\pi y), \; (x,y)\in [0,1]^2.
$$
\end{itemize}
\end{example}


\begin{figure}[H]
\begin{center}
\includegraphics[width=.8\textwidth,height=8cm]{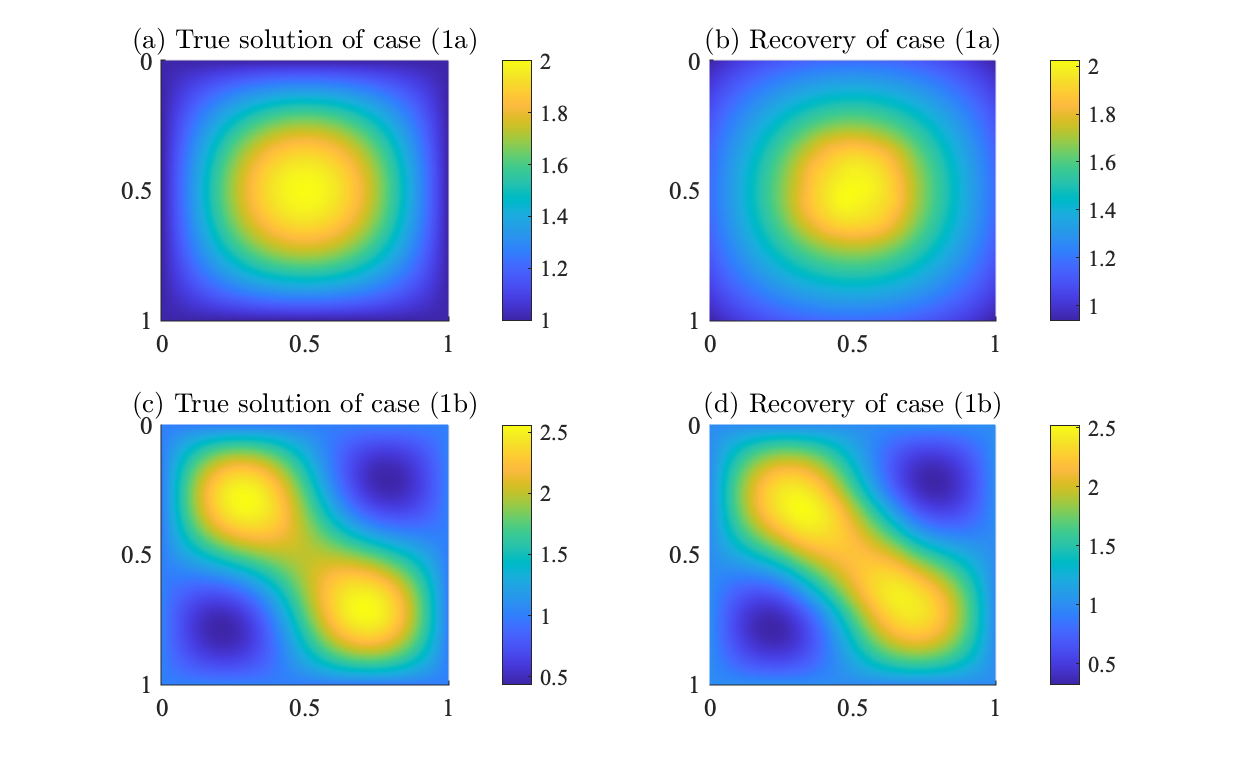}
\caption{The exact and reconstructed sources from interior observation for Example \ref{ex1}.}\label{fig-ex1}
\end{center}
\end{figure}

Set the initial guesses $f_0=1$ for case (1a) and $f_0=1+0.75\sin(2\pi x)\sin(2\pi y)$ for case (1b), respectively. Fix the noise level $\delta=5\%$. Figure \ref{fig-ex1} shows the recovered solutions of cases (1a) and (1b),  respectively. We see that smooth sources can be well reconstructed by solving the Tikhonov regularization functional \eqref{L^2-term}. Furthermore, the results in Figure \ref{fig-ex1} indicate that the uniqueness theorem \ref{uniqueness-interior} can be applied in practical experiments, namely, the partial interior observation are able to support the satisfactory recoveries.


Next, we investigate the numerical inversion for the nonsmooth sources. However, the classical $L^2$ penalty term in Tikhonov regularization has the tendency to over-smooth solutions, which makes it difficult to capture special features of the sought solution such as sparsity and discontinuous. To overcome this drawback, we will employ a regularization scheme in a manner of incorporating general non-smooth convex penalty terms, which has been applied in \cite{JinWang, Zhong}. In the follows, we give a sketch for solving \eqref{operator-G}, which essentially consists of considering
\begin{equation*}
\min_{f\in L^2(\Omega)} \mathcal{\widetilde J}(f), \quad \min_{f\in L^2(\Omega)} \mathcal{R}(f)
\end{equation*}
simultaneously in some balanced way, where the date-match functional $\mathcal{J}(f)$ is given by 
\begin{equation*}
\mathcal{\widetilde J}(f):=\frac{1}{2}\|\mathbb{G}(f)-g^\delta\|_2^2,
\end{equation*}
and $\mathcal{R}$ is the $L^2$ combined with TV penalty term given by
\begin{equation*}
\mathcal{R}(f)=\frac{1}{2\kappa}\|f\|_2^2 +|f|_{\rm TV}.
\end{equation*}  
Pick $f_{-1}=f_0:=f_0\in L^2(\Omega)$ and $\xi_{-1}=\xi_0:=\xi_0\in \partial \mathcal{R}(f_0)$ as initial guesses, where $\partial\mathcal{R}(f):=\big\{\xi\in L^2(\Omega): \mathcal{R}(\tilde f)-\mathcal{R}(f)- (\xi,\tilde f-f) \ge 0 \ \text{for all} \ \tilde f\in L^2(\Omega)\big\}$. For $n\ge 0$, we define 
\begin{equation}\label{TPG-algorithm}
\begin{cases}
\begin{aligned}
\eta_n&=\xi_n+\lambda_n(\xi_n-\xi_{n-1}),\\
z_n&=\arg \min_{z\in L^2(\Omega)} \big\{\mathcal{R}(z)-(\eta_n, z)\big\},\\
\xi_{n+1}&=\xi_n+\gamma_n \mathbb{G}^*(g^\delta-\mathbb{G}(z_n)),\\
f_{n+1}&=\arg \min_{f\in L^2(\Omega)} \big\{\mathcal{R}(f)-(\xi_{n+1},f)\big\},\\
\end{aligned}
\end{cases}
\end{equation}
where $\lambda_n\ge 0$ is the combination parameter, $\gamma_n$ is the step sizes defined by
\begin{equation*}
\gamma_n=
\begin{cases}
\min\left\{\frac{\bar\gamma_0\|\mathbb{G}(z_n^\delta)-g^\delta\|_2^2}{\|\mathbb{G}^*(\mathbb{G}(z_n^\delta)-g^\delta) \|_2^2}, \bar\gamma_1\right\} &\text{if} \ \|\mathbb{G}(z_n^\delta)-g^\delta\|_2>\tau\delta\\
0 &\text{if} \ \|\mathbb{G}(z_n^\delta)-g^\delta\|_2 \leq\tau\delta
\end{cases}
\end{equation*}
for some positive constants $\bar\gamma_0$ and $\bar\gamma_1$.  The combination parameter $\lambda_n\ge 0$ satisfies $\lambda_0=0, \ \lambda_n\in [0,1]$ for all $n\in \mathbb{N}$ and is chosen by $n/(n+5)$. Further, the iteration \eqref{TPG-algorithm} will be stopped by the discrepancy principle with respect to $z_n$ i.e., for a given $\tau>1$, we will terminate the iteration after $n_*$ steps, where $n_*:=n(\delta,g^\delta)$ is the integer such that
\begin{equation}\label{stop-rule}
\|\mathbb{G}(z_{n_*})-g^\delta\|_2 \leq \tau\delta <\|\mathbb{G}(z_{n_*-1})-g^\delta\|_2, \quad 0\leq n<n_*.
\end{equation}
Now we are ready to show the performance of the above iterative algorithm to recover the unknown sources. The primal dual hybrid gradient method will be applied to solve \eqref{TPG-algorithm} iteratively. We set $\tau=1.05$ in (\ref{stop-rule}) to stop the iteration.

\begin{example}\label{ex2}
Consider the numerical inversion for following cases:
\begin{itemize}
\item[(2a):]  the non-smooth solution is given by
$$
f_{\rm true}(x,y)=\begin{cases}
2, & (x-0.5)^2+(y-0.5)^2\leq 0.25^2,\\
0, &\text{\rm else}; 
\end{cases}
$$
\item[(2b):] the non-smooth solution is given by
$$
f_{\rm true}(x,y)=\begin{cases}
3, & (x-0.3)^2+(y-0.3)^2\leq 0.15^2,\\
3, & (x-0.7)^2+(y-0.7)^2\leq 0.15^2,\\
0, &\text{\rm else}.
\end{cases}
$$
\end{itemize}
\end{example}

Set the initial guess $f_0=0$ and the noise level $\delta=1\%$. We plot the reconstructions of cases (2a) and (2b) in Figure \ref{fig-ex2}, respectively. We see that the reconstructed source can detect the shapes and positions of the inclusions. However, the recovered source suffers from oversmoothing near the inclusion edges, which leads to the interface of inclusion not being captured well. Moreover, we can see from Figure \ref{fig-ex2}, (d) that the side of inclusions far away from the interior observation area appears to be difficult to accurately recover.

\begin{figure}[H]
\begin{center}
\includegraphics[width=.8\textwidth,height=8cm]{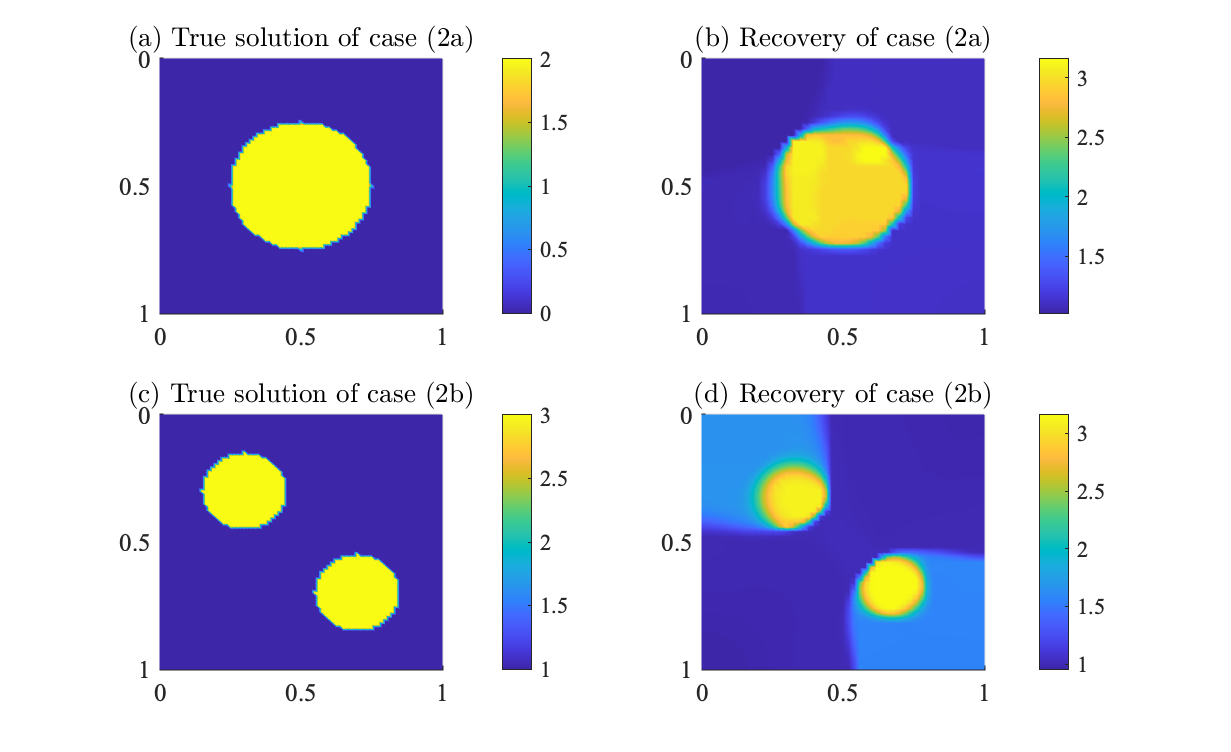}
\caption{The exact and reconstructed sources from interior observation for Example \ref{ex2}.}\label{fig-ex2}
\end{center}
\end{figure}


\section*{Acknowledgments}
This work was partially supported by the National Natural Science Foundation of China (No. 12301555, No. 12271277, No. 12201298), the National Key R\&D Program of China (No. 2023YFA1008903), the Natural Science Foundation of Jiangsu Province (No. BK20210269), the Taishan Scholars Program of Shandong Province (No. tsqn202306083), and the Open Research Fund of Key Laboratory of Nonlinear Analysis \& Applications (Central China Normal University), Ministry of Education, China. The first author also thanks Ningbo Youth Leading Talent Project (No.2024QL045).






\bibliographystyle{plainurl} 

\end{document}